\documentclass[12pt]{article}
\usepackage{graphicx}
\usepackage{amsmath}
\usepackage{amsfonts}
\usepackage{amsthm}
\usepackage[T1]{fontenc}
\usepackage{url}
\usepackage{color}
\usepackage[margin=1in]{geometry}
\usepackage{amssymb,bm}
\usepackage{amsmath}
\usepackage{amsthm}
\usepackage{graphicx}
\usepackage[active]{srcltx} 
\usepackage{hyperref}
\hypersetup{pdfborder=0 0 0}

\newtheorem{theorem}{{\sc Theorem}}[section]

\def\Xint#1{\mathchoice
{\XXint\displaystyle\textstyle{#1}}%
{\XXint\textstyle\scriptstyle{#1}}%
{\XXint\scriptstyle\scriptscriptstyle{#1}}%
{\XXint\scriptscriptstyle\scriptscriptstyle{#1}}%
\!\int}
\def\XXint#1#2#3{{\setbox0=\hbox{$#1{#2#3}{\int}$ }
\vcenter{\hbox{$#2#3$ }}\kern-.6\wd0}}

\def\dashint{\Xint-}

\newcommand{\Gk}{\kappa}

\newcommand{\Gth}{\theta}

\bmdefine\BGa{\alpha}
\bmdefine\BGb{\beta}
\bmdefine\BGd{\delta}
\bmdefine\BGe{\epsilon}
\bmdefine\BGve{\varepsilon}
\bmdefine\BGf{\phi}
\bmdefine\BGvf{\varphi}
\bmdefine\BGg{\gamma}
\bmdefine\BGc{\chi}
\bmdefine\BGi{\iota}
\bmdefine\BGk{\kappa}
\bmdefine\BGl{\lambda}
\bmdefine\BGn{\eta}
\bmdefine\BGm{\mu}
\bmdefine\BGv{\nu}
\bmdefine\BGp{\pi}
\bmdefine\BGth{\theta}
\bmdefine\BGvth{\vartheta}
\bmdefine\BGr{\rho}
\bmdefine\BGvr{\varrho}
\bmdefine\BGs{\sigma}
\bmdefine\BGvs{\varsigma}
\bmdefine\BGt{\tau}
\bmdefine\BGj{\tau}
\bmdefine\BGu{\upsilon}
\bmdefine\BGo{\omega}
\bmdefine\BGx{\xi}
\bmdefine\BGy{\psi}
\bmdefine\BGz{\zeta}
\bmdefine\BGD{\Delta}
\bmdefine\BGF{\Phi}
\bmdefine\BGG{\Gamma}
\bmdefine\BGL{\Lambda}
\bmdefine\BGP{\Pi}
\bmdefine\BGT{\Theta}
\bmdefine\BGS{\Sigma}
\bmdefine\BGU{\Upsilon}
\bmdefine\BGO{\Omega}
\bmdefine\BGX{\Xi}
\bmdefine\BGY{\Psi}

\bmdefine\BCA{{\mathcal A}}
\bmdefine\BCB{{\mathcal B}}
\bmdefine\BCC{{\mathcal C}}
\bmdefine\BCD{{\mathcal D}}
\bmdefine\BCE{{\mathcal E}}
\bmdefine\BCF{{\mathcal F}}
\bmdefine\BCG{{\mathcal G}}
\bmdefine\BCH{{\mathcal H}}
\bmdefine\BCI{{\mathcal I}}
\bmdefine\BCJ{{\mathcal J}}
\bmdefine\BCK{{\mathcal K}}
\bmdefine\BCL{{\mathcal L}}
\bmdefine\BCM{{\mathcal M}}
\bmdefine\BCN{{\mathcal N}}
\bmdefine\BCO{{\mathcal O}}
\bmdefine\BCP{{\mathcal P}}
\bmdefine\BCQ{{\mathcal Q}}
\bmdefine\BCR{{\mathcal R}}
\bmdefine\BCS{{\mathcal S}}
\bmdefine\BCT{{\mathcal T}}
\bmdefine\BCU{{\mathcal U}}
\bmdefine\BCV{{\mathcal V}}
\bmdefine\BCW{{\mathcal W}}
\bmdefine\BCX{{\mathcal X}}
\bmdefine\BCY{{\mathcal Y}}
\bmdefine\BCZ{{\mathcal Z}}

\bmdefine\Bzr{ 0}
\bmdefine\Ba{ a}
\bmdefine\Bb{ b}
\bmdefine\Bc{ c}
\bmdefine\Bd{ d}
\bmdefine\Be{ e}
\bmdefine\Bf{ f}
\bmdefine\Bg{ g}
\bmdefine\Bh{ h}
\bmdefine\Bi{ i}
\bmdefine\Bj{ j}
\bmdefine\Bk{ k}
\bmdefine\Bl{ l}
\bmdefine\Bm{ m}
\bmdefine\Bn{ n}
\bmdefine\Bo{ o}
\bmdefine\Bp{ p}
\bmdefine\Bq{ q}
\bmdefine\Br{ r}
\bmdefine\Bs{ s}
\bmdefine\Bt{ t}
\bmdefine\Bu{ u}
\bmdefine\Bv{ v}
\bmdefine\Bw{ w}
\bmdefine\Bx{ x}
\bmdefine\By{ y}
\bmdefine\Bz{ z}
\bmdefine\BA{ A}
\bmdefine\BB{ B}
\bmdefine\BC{ C}
\bmdefine\BD{ D}
\bmdefine\BE{ E}
\bmdefine\BF{ F}
\bmdefine\BG{ G}
\bmdefine\BH{ H}
\bmdefine\BI{ I}
\bmdefine\BJ{ J}
\bmdefine\BK{ K}
\bmdefine\BL{ L}
\bmdefine\BM{ M}
\bmdefine\BN{ N}
\bmdefine\BO{ O}
\bmdefine\BP{ P}
\bmdefine\BQ{ Q}
\bmdefine\BR{ R}
\bmdefine\BS{ S}
\bmdefine\BT{ T}
\bmdefine\BU{ U}
\bmdefine\BV{ V}
\bmdefine\BW{ W}
\bmdefine\BX{ X}
\bmdefine\BY{ Y}
\bmdefine\BZ{ Z}

\date{}
\begin{document}
\title{The asymptotically sharp geometric rigidity interpolation estimate in thin bi-Lipschitz domains}
\author{D. Harutyunyan}
\maketitle

\begin{abstract}
This work is part of a program of development of asymptotically sharp geometric rigidity estimates for thin domains. A thin domain in three dimensional Euclidean space is roughly a small neighborhood of regular enough two dimensional compact surface. We prove an asymptotically sharp geometric rigidity interpolation inequality for thin domains with little regularity.
In contrast to that celebrated Friesecke-James-M\"uller rigidity estimate [\textit{Comm. Pure Appl. Math.,} 55(11):1461-1506, 2002] for plates, our estimate holds for any proper rotations 
$\BR\in SO(3).$ Namely, the estimate bounds the $L^p$ distance of the gradient of any $\By\in W^{1,p}(\Omega,\mathbb R^3)$ field from any constant proper rotation $\BR\in SO(3)$, in terms of the average $L^p$ distance (nonlinear strain) of the gradient $\nabla\By$ from the rotation group $SO(3)$, and the average $L^p$ distance of the field itself from the set of rigid motions corresponding to the rotation $\BR$. There are several remarkable facts about the estimate: 1. The constants in the estimate are sharp in terms of the domain thickness scaling for any thin domains with the required regularity. 2. In the special cases when the domain has positive or negative Gaussian curvature, the inequality reduces the problem of estimating the gradient $\nabla\By$ in terms of the prototypical nonlinear strain $\int_\Omega\mathrm{dist}^p(\nabla\By(x),SO(3))dx$ to the easier problem of estimating only the vector field $\By$ in terms of the nonlinear strain without any loss in the constant scalings as the Ans\"atze suggest. The later will be a geometric rigidity Korn-Poincar\'e type estimate. This passage is major progress in the thin domain rigidity problem. 3. For the borderline energy scaling (bending-to-stretching), the estimate implies improved strong compactness on the vector fields for free. Finally, this being said, our new interpolation inequality reduces the problem of proving "any" geometric one well rigidity problem in thin domains to estimating the vector field itself instead of the gradient, thus reducing the complexity of the problem. 
\end{abstract}

\section{Introduction}
\label{sec:1}

Let $S\subset \mathbb R^3$ be a connected and compact surface that has a normal a.e. (one may assume for instance that $S$ is Lipschitz). Given a small parameter $h>0,$ recall that a shell $S^h$ of thickness $h$ is the $h/2$ neighborhood of $S$ in the normal direction, i.e., $S^h=\{x+t\Bn(x) \ : \ x\in S,\ t\in (-h/2,h/2)\},$ where for any point $x\in S,$ the vector $\Bn(x)$ is the unit normal to $S$ at $x.$ Thin spatial domains are roughly shells with non-constant thickness. Namely, assume the functions $g_1^h(x),g_2^h(x)\colon S\to (0,\infty)$ are of order $h$ Lipschitz functions, i.e., they fulfill the uniform conditions 
\begin{equation}
\label{1.1}
h\leq g_1^h(x),g_2^h(x)\leq c_1 h,\quad \text{and}\quad |\nabla g_1^h(x)|+|\nabla g_2^h(x)|\leq c_2h,\quad\text{for a.e.}\quad x\in S,
\end{equation}
for some fixed constants $c_1,c_2>0.$ Then the set $\Omega$ given by 
\begin{equation}
\label{1.2}
\Omega=\{x+t\Bn(x) \ : \ x\in S,\ t\in (-g_1^h(x),g_2^h(x))\},
\end{equation}
will be a thin domain if the normal $\Bn(x)$ exists for all $x\in S.$ If $\Bn(x)$ exists only a.e. in $S,$ then instead of $\Omega$ we can consider the interior of the closure $\bar \Omega^{\circ},$ which will have no holes for any $x\in S^{\circ}.$ This being said we can assume without loss of generality that $S$ has a normal at every point $x\in S$ and that $g_1^h$ and $g_2^h$ are differentiable at every $x\in S.$ The problem of determination of the geometric rigidity (which will be defined below) of thin domains is a central task in nonlinear elasticity. The geometric rigidity of a thin domain can be defined through the celebrated geometric rigidity estimate of Friesecke, James and M\"uller [\ref{bib:Fri.Jam.Mue.1}], which reads as follows: \textit{Assume $\Omega\subset\mathbb R^3$ is open bounded connected and Lipschitz. Then there exists a constant $C=C(\Omega)$ depending only on $\Omega,$ such that for every vector field $\By\in H^1(\Omega),$ there exists a constant proper rotation $\BR\in SO(3)$ such that}
\begin{equation}
\label{1.3}
\|\nabla\By-\BR\|_{L^2(\Omega)}^2\leq C\int_\Omega\mathrm{dist}^2(\nabla\By(x),SO(3))dx.
\end{equation}
It is known that for thin domains $\Omega,$ the constant $C$ in (\ref{1.3}) blows up as $h$ goes to zero [2,3,4,5,6,7,8,9,10,11,12,13,21,22,23], and it typically has the asymptotic form $C=ch^{\alpha}$ for some $c>0$ and $\alpha<0.$ The geometric rigidity of $\Omega$ is then the exponent $\alpha$ in $C=ch^{\alpha}$ provided it exists as $h\to 0$; the bigger $\alpha$ is, the more rigid the domain is. A large amount of evidence in the literature suggests the yet unproven fact that $\alpha$ should depend only on the domain mid-surface $S$, while the constant $c$ will depends also on the constants $c_1$ and $c_2$ in (\ref{1.1}). In particular, for the case when $S$ has positive or negative Gaussian curvature the author has proven in [\ref{bib:Harutyunyan.2}] that $\alpha=1$ or $\alpha=4/3$ respectively for the around identity linearized analogue of (\ref{1.3}). Let us provide somewhat more details on the later. If one rotates the field $\Bu$ by the rotation $\BR$ in (\ref{1.3}), then the resulting estimate shows that one can assume without loss of generality that $\BR=\BI$. Recall that the field $\By-\Bx$ is the displacement in nonlinear elasticity, thus for small enough deformations, one has the equality $\nabla\By=\BI+\nabla\Bu,$ where the quantity $\nabla\Bu$ is small.
Thus we have the linearization
$$\mathrm{dist}(\nabla\Bu(x)+\BI,SO(3))=(\nabla\Bu+\nabla\Bu^T)=2e(\Bu),$$
where $e(\Bu)$ is the linear strain in the linear elasticity theory. Hence, upon this linearization one arrives at a new inequality, called Korn's first inequality
\begin{equation}
\label{1.4}
\|\nabla\Bu\|_{L^2(\Omega)}^2\leq C\|e(\Bu)\|_{L^2(\Omega)}^2.
\end{equation}
The inequality (\ref{1.4}) has been introduced by Korn [\ref{bib:Korn.1},\ref{bib:Korn.2}] to prove coerciveness of the linear elastic energy, see also 
[\ref{bib:Friedrichs},\ref{bib:Horgan},\ref{bib:Kon.Ole.1},\ref{bib:Kon.Ole.2}] for more details. Certainly (\ref{1.4}) follows from (\ref{1.3}) with the same constant $C$ by taking $\By=\Bx+\epsilon\Bu$ and letting $\epsilon$ go to zero. The problem of determining the rigidity of a given thin domain (the asymptotics of $C$ in terms of the thickness parameter $h$ in (\ref{1.3})) has been solved for plates\footnote{Or for shells that have a flat part} by Friesecke, James and M\"uller  [\ref{bib:Fri.Jam.Mue.1}], yielding the result $C=ch^{-2}.$ That is when
$\Omega=\omega\times(-h/2,h/2)$ for some connected compact Lipschitz set $\omega\subset\mathbb R^2.$ Also, it has been shown in [\ref{bib:Fri.Jam.Mor.Mue.}] that in fact the estimate (\ref{1.3}) holds with $\alpha=-2$ for any shells $S^h,$ i.e., there exist constant $C,h_0>0$ depending only on $S$ such that 
\begin{equation}
\label{1.5}
\|\nabla\By-\BR\|_{L^2(S^h)}^2\leq \frac{C}{h^2}\int_{S^h}\mathrm{dist}^2(\nabla\By(x),SO(3))dx.
\end{equation}
for all $h\in (0,h_0).$ Among others, very significant applications of the inequality (\ref{1.3})
are that it allows one to derive shell theories from three dimensional elasticity for certain scaling regimes of the elastic energy in terms of the thickness $h,$ 
[\ref{bib:Fri.Jam.Mue.1},\ref{bib:Fri.Jam.Mue.2},\ref{bib:Fri.Jam.Mor.Mue.},\ref{bib:Hor.Lew.Pak.},\ref{bib:Lew.Mor.Pak.}], as well as it allows one to calculate 
the critical buckling load in shell buckling problems [\ref{bib:Gra.Har.3},\ref{bib:Gra.Har.4}]. In both problems knowing the the asymptotics of the optimal constant in (\ref{1.3})
is crucial. 
In this work we are concerned with studying (\ref{1.3}) for thin domains $\Omega\subset\mathbb R^3.$ We prove an interpolation version of (\ref{1.3}) which reduces the problem to the estimation of the deviation of the vector field $\Bu$ itself (not the gradient) from the group of rigid motions. Despite the fact that the proof of the inequality is surprisingly elementary and easy, this is apparently a significant reduction of the complexity of the problem, taking into account the fact that in the case of uniformly positive or negative Gaussian curvature thin domains (this refers to the Gaussian curvature of the mid-surface $S$), no loss of the constant in terms of the asymptotics in $h$ is expected as the lower bounds and the Ans\"atze in [\ref{bib:Harutyunyan.2}] suggest. We provide a more detailed observation on this in the next section after the formulation of the main result.

\section{Main result}
\setcounter{equation}{0}

For the mid-surface $S$ to have a normal a.e. and the main result to hold, we will impose some mild regularity condition on $S.$ While in the literature a usual assumption 
on the mid-surface $S$ would be $C^2$, we will assume that $S$ is compact connected and bi-Lipschitz, i.e., it has a finite atlas with bi-Lipschitz patches. 
This will imply the following geometric condition: \textit{There exist $\sigma,\delta\in (0,1)$ such that for every $r\in (0,\delta)$ and every $x\in S$ one has 
\begin{equation}
\label{2.1}
\frac{\mathcal {H}^2 (B_{r}(x)\cap S)}{\mathcal {H}^2 (B_{2r}(x)\cap S)}\leq \sigma,
\end{equation}
where $\mathcal {H}^2$ is the two-dimensional Hausdorff measure (surface measure in this case).} Condition (\ref{2.1}) roughly means that there can be no infinitesimal local concentrations of the surface $S.$

\begin{theorem}
\label{th:2.1}
Let the surface $S\subset\mathbb R^3$ be compact connected and bi-Lipschitz, and let $h>0$ be a small parameter. Assume the family of functions $g_1^h,g_2^h \colon S\to (0,\infty)$ fulfills the uniform conditions 
\begin{equation}
\label{2.2}
h\leq g_1^h(x),g_2^h(x)\leq c_1 h,\quad \text{and}\quad |\nabla g_1^h(x)|+|\nabla g_2^h(x)|\leq c_2h,\quad\text{for all}\quad x\in S,
\end{equation}
for some constants $c_1,c_2>0,$ and denote the thin domains 
$$
\Omega^h=\{x+t\Bn(x) \ : \ x\in S,\ t\in (-g_1^h(x),g_2^h(x))\}.
$$
Let $1<p<\infty$ and let $\|\cdot\|_{p}$ denote the $L^p(\Omega^h)$ norm. Then there exists a constants $C,h_0>0,$ depending only on $S$ and the constants $c_1,c_2>0$ such that for any vector field $\By\in W^{1,p}(\Omega^h),$ any proper rotation $\BR\in SO(3),$ and any constant vector $\Bb\in\mathbb R^3$ one has the estimate 
\begin{equation}
 \label{2.3}
 \|\nabla \By-\BR\|_p^2\leq C\left(\frac{\|\By-\BR x-\Bb\|_p\|\mathrm{dist}(\nabla\By,SO(3))\|_p}{h}+\|\By-\BR x-\Bb\|_p^2+\|\mathrm{dist}(\nabla\By,SO(3))\|_p^2\right),
\end{equation}
for all $h\in (0,h_0).$ Moreover, if in addition $S$ has a $C^2$ piece that admits a local chart in terms of the principal coordinates, then the exponent of $h$ in the inequality (\ref{2.2}) is optimal for $\Omega^h,$ i.e., there exists a deformation 
$\By\in W^{1,p}(\Omega^h,\mathbb R^3)$ realizing the asymptotics of $h$ in (\ref{2.2}).
\end{theorem}

Some remarks are in order. 
\begin{itemize}
\item[1.] Note first that for any given displacement $\Bu\in W^{1,p}(\Omega^h,\mathbb R^3),$ taking the sequence of deformations $\By^\epsilon=\Bx+\epsilon\Bu$ and then letting $\epsilon$ go to zero we derive from (\ref{2.3}) the estimate 
\begin{equation}
 \label{2.4}
 \|\nabla \Bu\|_p^2\leq C\left(\frac{\|\Bu\|_p\|e(\Bu)\|_p}{h}+\|\Bu\|_p^2+\|e(\Bu)\|_p^2\right),
\end{equation}
where $\Bu=\frac{1}{2}(\nabla\Bu+\nabla\Bu^T)$ is the linear strain. The estimate (\ref{2.4}) is the linear version of (\ref{2.3}). A stronger version of (\ref{2.4}) has been proven in 
[\ref{bib:Harutyunyan.3},\ref{bib:Harutyunyan.4}] for the $L^2$ norm and in [\ref{bib:Harutyunyan.5}] for every $1<p<\infty$, where in place of the product $\|\Bu\|_2\|e(\Bu)\|_2$ one has $\|\Bn\cdot \Bu\|_2\|e(\Bu)\|_2,$ i.e., only the out-of-plane component of the field enters the estimate. 
\item[2.] Next denote the Gaussian curvature of $S$ by $K.$ Tovstik and Smirnov have constructed an Ansatz [\ref{bib:Tov.Smi.}] that realizes the asymptotics $\alpha=-4/3$ in the constant $C=ch^{\alpha}$ in (\ref{1.3}) in the case $K<0.$ Also the author has constructed an Ansatz [\ref{bib:Harutyunyan.2}] that gives 
the asymptotics $\alpha=-1$ in (\ref{1.3}) in the case $K>0.$ Furthermore, it has been proven in [\ref{bib:Harutyunyan.2}] that if zero boundary condition is imposed 
on the vector field $\Bu$ on the thin face of the thin domain $\Omega^h,$ then in the linear version (\ref{1.4}) 
one indeed has $\alpha=-4/3$ and $\alpha=-1$ in the cases $K<0$ and $K<0$ respectively. Also, it has been shown [\ref{bib:Harutyunyan.2}] that 
$$\|\Bu\|_{L^2(\Omega^h)}^2\leq Ch^\beta\|e(\Bu)\|_{L^2(\Omega)}^2,$$
where $\beta=-1/3$ for $K<0$ and $\beta=0$ for $K>0.$ This implies that the estimate (\ref{2.3}) indeed does not suffer an asymptotic loss of constants 
at least in the cases $K><0,$ and thus from this point on, the interpolation estimate (\ref{2.3}) may be utilized for the purpose of proving asymptotically optimal 
rigidity estimates on gradient fields.
\item[3.] In the derivation of shell theories from nonlinear elasticity by $\Gamma-$convergence, one usually assumes that the elastic energy density has a quadratic growth at the group of proper rotations (which is the elastic energy well) [\ref{bib:Fri.Jam.Mue.1},\ref{bib:Fri.Jam.Mue.2},\ref{bib:Hor.Lew.Pak.},\ref{bib:Lew.Mor.Pak.},\ref{bib:Mueller}], i.e., 
$$W(\nabla\By)\geq C\cdot\mathrm{dist}^2(\nabla\By,SO(3))$$
for some fixed constant $C>0$ and all vector fields $\By\in W^{1,2}(\Omega,\mathbb R^3).$ Then one considers different scaling regimes of the elastic energy $\int_\Omega^h W(\nabla\By(x))dx$
in terms of the thickness $h.$ A critical energy scaling is $E_{el}\sim h^3,$ which corresponds to bending. It is then worth mentioning that while Friesecke-James-M\"uller estimate (\ref{1.3}) 
for $\alpha=-2$ implies weak relative compactness in $L^2$ for the sequence of the rescaled (in the normal variable) gradients\footnote{However there are tools to deduce strong convergence of the gradients} $\nabla\By_h$ and strong compactness in $L^2$ for the sequence of rescaled fields $\By_h$ (as $h\to 0$), our estimate (\ref{2.3}) then will clearly yield strong compactness of the rescaled gradients $\nabla\By_h$ immediately. 
\end{itemize}

\section{Proof of the main result}
\label{sec:3}
\setcounter{equation}{0}

\begin{proof}[Proof of Theorem~\ref{th:2.1}]
We divide the proof into several steps for the convenience of the reader.\\
1. We first somewhat simplify the estimate (\ref{2.3}). First of all a translation by a fixed vector $\By=\Bv+\Bb$ does not change the gradient, thus we can assume without loss of generality that $\Bb=0.$ Next, denoting $\Bv=\BR\Bw,$ the left hand side of (\ref{2.3}) will become $\|\Bw-\BI\|_p^2,$ and the right hand side of (\ref{2.3}) will remain the same expression written out for $\Bw$ in place of $\Bv.$ This being said we can assume without loss of generality that $\BR=\BI$ and $\Bb=0$ in (\ref{2.3}). Finally, making a change of variables $\By=\Bw+\Bx$ will transform the new form of (\ref{2.3}) to the estimate 
\begin{equation}
 \label{3.1}
 \|\nabla \Bw\|_p^2\leq C\left(\frac{\|\Bw\|_p\|\mathrm{dist}(\nabla\Bw+\BI,SO(3))\|_p}{h}+\|\Bw\|_p^2+\|\mathrm{dist}(\nabla\Bw+\BI,SO(3))\|_p^2\right),
\end{equation} 
to be now proven.\\
2. In the second step we prove the following statement: \textit{Under the conditions of Theorem~\ref{th:2.1}, the estimate (\ref{3.1}) holds if and inly if one has for any field 
$\Bv\in W^{1,p}(\Omega^h,\mathbb R^3)$ the estimate
\begin{equation}
 \label{3.2}
 \|\nabla \Bv\|_p^2\leq C_1\left(\frac{\|\Bv\|_p^2}{h^{t}}+\frac{\|\mathrm{dist}(\nabla\Bv+\BI,SO(3))\|_p^2}{h^{2-t}}\right),
\end{equation}
for any $t\in [0,2].$ Here $C_1>0,$ and $C$ and $C_1$ in (\ref{3.1}) and (\ref{3.2}) are of the same order, namely, $1/2\leq\frac{C_1}{C}\leq 2.$ }\\
Evidently as $h>0$ is small, we have $h^t,h^{2-t}\leq 1$ for $t\in[0,2]$ and (\ref{3.1}) implies (\ref{3.2}) by the arithmetic-geometric mean inequality with $C_1=\frac{3}{2}C.$ 
Assume now (\ref{3.2}) holds. Given the fixed vector field $\Bw\in W^{1,p}(\Omega^h,\mathbb R^3),$ if 
$$\frac{\|\Bw\|_p^2}{h^{t_0}}=\frac{\|\mathrm{dist}(\nabla\Bw+\BI,SO(3))\|_p^2}{h^{2-t_0}}\quad \text{for some} \quad t_0\in[0,2],$$
 then we choose $t=t_0$ in (\ref{3.2}) and get the estimate 
$$\|\nabla \Bw\|_p^2\leq 2C_1\frac{\|\Bw\|_p\|\mathrm{dist}(\nabla\Bw+\BI,SO(3))\|_p}{h}.$$
If 
$$\frac{\|\Bw\|_p^2}{h^{t}}<\frac{\|\mathrm{dist}(\nabla\Bw+\BI,SO(3))\|_p^2}{h^{2-t}}\quad \text{ for all} \quad t\in[0,2],$$
then (\ref{3.2}) implies 
$$\|\nabla \Bw\|_p^2\leq 2C_1\mathrm{dist}(\nabla\Bw+\BI,SO(3))\|_p^2,$$
which in tern yields (\ref{3.1}) with $C=2C_1.$ The case 
$$\frac{\|\Bw\|_p^2}{h^{t}}<\frac{\|\mathrm{dist}(\nabla\Bw+\BI,SO(3))\|_p^2}{h^{2-t}}\quad  \text{for all} \quad t\in[0,2]$$
is analogous.\\
3. Now we focus our attention to the simplified estimate (\ref{3.2}) with no product terms. We will first prove it on the shell $S^h=\{x+t\Bn(x) \ : \ x\in S,\ t\in (-h/2,h/2)\}$ with thickness $h$ around $S$ and then pass to the thin domain $\Omega^h$ employing a localization argument originated in [\ref{bib:Koh.Vog.}] and successfully employed in the derivation of the estimate (\ref{1.5}), as well as 
in [\ref{bib:Harutyunyan.3}] to pass from shells to thin domains. Fix $\gamma\in[0,1]$ and divide the shell $S^h$ into small compact shells with in-plane size 
of order $h^\gamma.$ Denoting $m=[1/h^\gamma],$ we have $N=O(m^2)$ shells $S_1^h,S_2^h,\dots,S_N^h,$ with thickness $h$ and in-plane size 
roughly $h^\gamma.$ It is known that the estimate (\ref{1.5}) holds in $L^p$ for all $1<p<\infty,$ thus if we scale each $S_i^h$ by $h^\gamma$ to get a new shell 
$\frac{1}{h^\gamma}S_i^h$ with in-plane size of order one and thickness $h^{1-\gamma},$ and apply the estimate (\ref{1.5}) to the fields $h^\gamma(\Bv+\Bx),$ 
we obtain for the vector field $\Bv+\Bx$ the estimate
\begin{equation}
\label{3.3}
\|\nabla\Bv+\BI-\BR_i\|_{L^p(S_i^h)}\leq \frac{C}{h^{1-\gamma}}\|\mathrm{dist}(\nabla\Bv+\BI,SO(3))\|_{L^p(S_i^h)} ,\quad i=1,2,\dots,N,
\end{equation} 
for some local rotations $\BR_i\in SO(3)$ and some uniform constant $C>0$ that depends only on $S.$  
Consequently we obtain from (\ref{3.3}) the bound
\begin{equation}
\label{3.4}
\|\nabla\Bv\|_{L^p(S_i^h)}\leq \|\BI-\BR_i\|_{L^p(S_i^h)}+\frac{C}{h^{1-\gamma}}\|\mathrm{dist}(\nabla\Bv+\BI,SO(3))\|_{L^p(S_i^h)},   \quad i=1,2,\dots,N.
\end{equation} 
For the average fields $\Bb_i=\dashint_{S_i^h}(\Bv(x)+(\BI-\BR_i)x)dx$ we have by the Poincar\'e inequality and (\ref{3.3}) that
\begin{align}
\label{3.5}
\|\Bv+(\BI-\BR_i)x-\Bb_i\|_{L^p(S_i^h)}& \leq Ch^{\gamma} \|\nabla \Bv+\BI-\BR_i\|_{L^p(S_i^h)}  \\ \nonumber
&\leq Ch^{2\gamma-1}\|\mathrm{dist}(\nabla\Bv+\BI,SO(3))\|_{L^p(S_i^h)},
\end{align} 
which gives the bound 
\begin{equation}
\label{3.6}
 \|(\BI-\BR_i)x-\Bb_i\|_{L^p(S_i^h)}\leq \|\Bv\|_{L^p(S_i^h)}+Ch^{2\gamma-1}\|\mathrm{dist}(\nabla\Bv+\BI,SO(3))\|_{L^p(S_i^h)},\quad i=1,2,\dots,N.
\end{equation} 
Next we claim that
\begin{equation}
\label{3.7}
\|(\BI-\BR_i)x-\Bb_i\|_{L^p(S_i^h)}\geq Ch^{\gamma}\|\BI-\BR_i\|_{L^p(S_i^h)},
\end{equation} 
for some $C>0$ uniformly in $i=1,2,\dots,N.$ Indeed, on one hand we have the obvious equality
\begin{equation}
\label{3.8}
\|\BI-\BR_i\|_{L^p(S_i^h)}^p=|\BI-\BR_i|^p|S_i^h|,\quad i=1,2,\dots,N.
\end{equation} 
To estimate the other term we interpret it geometrically. Each local rotation $\BR_i$ rotates around a unit vector $\Bn_i\in\mathbb R^3,$ 
which means that the operator $\BR_ix-\Bb_i\colon\mathbb R^3\to\mathbb R^3$ projects onto the plane $\pi_i$ orthogonal to $\Bn_i,$
then rotates by $\BR_i$ inside $\pi_i,$ and then translates by the vector $-\Bb_i.$ Assume the plane $\pi_i$ is applied at the tip of the vector $\Bb_i.$ Note that as each shell $S_i^h$ 
has in-plane size of order $h^\gamma,$ then invoking condition (\ref{2.1}) (first uncsaling each piece by $h^\gamma$) we have that for some fixed 
$\tau\in(0,1)$ (uniform in $i$ and independent on $h$), at most half of $S_i^h$ (in terms of the measure) gets projected into the disc $D_i,$ 
which is centered at $\Bb_i,$ lies in $\pi_i$ and has radius $\tau h^\gamma.$ As in two dimensions one always has $|(\BI-\BR)x|=|\BI-\BR||x|$ for any 
rotation $\BR\in SO(2)$ and any vector $x\in \mathbb R^2,$ then taking into account the above observation, we can write the obvious estimate
\begin{equation}
\label{3.9}
\|(\BI-\BR_i)x-\Bb_i\|_{L^p(S_i^h)}^p\geq \frac{(\tau h^\gamma)^p|\BI-\BR_i|^p|S_i^h|}{2},\quad i=1,2,\dots,N. 
\end{equation}
The estimate (\ref{3.7}) immediately follows from (\ref{3.8}) and (\ref{3.9}).
Finally putting together (\ref{3.4}), (\ref{3.6}) and (\ref{3.7}) and summing up the obtained estimates in $i,$ we discover the bound
\begin{equation}
\label{3.10}
\|\nabla\Bv\|_{L^p(S^h)}\leq C\left(\frac{1}{h^\gamma} \|\Bv\|_{L^p(S^h)}+\frac{1}{h^{1-\gamma}}\|\mathrm{dist}(\nabla\Bv+\BI,SO(3))\|_{L^p(S^h)}\right),
\end{equation} 
which is equivalent to (\ref{3.2}) and finishes the Ansatz-free lower bound part of the Theorem for shells $S^h.$ Now we pass from $S^h$ to the thin domain $\Omega^h$ 
by a localization argument. Fix any $h>0$ small enough divide the mid-surface $S$ into $N$ small parts that are of order $h$ in two in-plane orthogonal directions, 
where $N$ is roughly $1/h^2.$ We can do the division so that all the small pieces of surface have uniformly bounded Lipschitz constants for all small enough $h>0.$ Then denote the $h$ neighborhood in the normal direction of each small piece of surface by $s_i^h$ and the part bounded between the functions $g_1^h$ and $g_2^h$ in the normal direction by $\omega_i^h$ for $i=1,2,\dots,N.$ For any $i=1,2,\dots,N$, the domains $s_i^h$ and $\omega_i^h$ have uniformly bounded Lipschitz constants, thus we have by the estimate (\ref{1.3})
\begin{align}
\label{3.11}
\|\nabla\Bv+\BI-\BR_i^1\|_{L^p(s_i^h)}&\leq C\|\mathrm{dist}(\nabla\Bv+\BI,SO(3))\|_{L^p(s_i^h)},\\ \nonumber
\|\nabla\Bv+\BI-\BR_i^2\|_{L^p(\omega_i^h)}&\leq C\|\mathrm{dist}(\nabla\Bv+\BI,SO(3))\|_{L^p(\omega_i^h)},
\end{align}
for some uniform constant $C>0$ and constant rotations $\BR_i^1$ and $\BR_i^2.$ As (\ref{2.2}) suggests $s_i^h$ is a subset of $\omega_i^h$ thus we have by the triangle inequality and from (\ref{3.11}) that 
\begin{equation}
\label{3.12}
\|\BR_i^1-\BR_i^2\|_{L^p(s_i^h)}\leq C\|\mathrm{dist}(\nabla\Bv+\BI,SO(3))\|_{L^p(\omega_i^h)}.
\end{equation}
Recalling the fact that $s_i^h$ and $\omega_i^h$ have uniformly bounded Lipschitz constants, the condition (\ref{2.2}) implies that they have uniformly comparable volumes too, 
thus we get from (\ref{3.11}), (\ref{3.12}) and the triangle inequality:
\begin{align*}
\|\nabla\Bv\|_{L^p(\omega_i^h)}&\leq \|\nabla\Bv+\BI-\BR_i^2\|_{L^p(\omega_i^h)}+\|\BI-\BR_i^2\|_{L^p(\omega_i^h)}\\
&\leq C\|\mathrm{dist}(\nabla\Bv+\BI,SO(3))\|_{L^p(\omega_i^h)}+C\|\BI-\BR_i^2\|_{L^p(s_i^h)}\\
&\leq C\|\mathrm{dist}(\nabla\Bv+\BI,SO(3))\|_{L^p(\omega_i^h)}+C\|\BI-\BR_i^1\|_{L^p(s_i^h)}+\|\BR_i^1-\BR_i^2\|_{L^p(s_i^h)}\\
&\leq C\|\mathrm{dist}(\nabla\Bv+\BI,SO(3))\|_{L^p(\omega_i^h)}+C\|\nabla\Bv+\BI-\BR_i^1\|_{L^p(s_i^h)}+C\|\nabla\Bv\|_{L^p(s_i^h)}\\
&\leq C\|\mathrm{dist}(\nabla\Bv+\BI,SO(3))\|_{L^p(\omega_i^h)}+C\|\nabla\Bv\|_{L^p(s_i^h)}.\\
\end{align*}
Consequently summing up in $1,2,\dots,N$ we arrive at 
 \begin{equation}
\label{3.13}
\|\nabla\Bv\|_{L^p(\Omega^h)}\leq C\|\mathrm{dist}(\nabla\Bv+\BI,SO(3))\|_{L^p(\Omega^h)}+\|\nabla\Bv\|_{L^p(S^h)},
\end{equation}
coupling which with (\ref{3.10}) we discover (\ref{3.2}), i.e., the estimate for the thin domain $\Omega^h.$ This completes the Ansatz-free lower bound part of the proof.
 
Recall that in the case when $S$ has patch that is $C^2$, then an Ansatz realizing the asymptotics of $h$ in Korn's first inequality for shells with positive Gaussian curvature has been constructed in [\ref{bib:Harutyunyan.2}]. It turns out that the same Ansatz also works for (\ref{2.3}). For the sake of completeness and convenience of the reader we recall the Ansatz construction here. Assume $S$ is $C^2$ and has a single patch given by the parametrization $\Br=\Br(\Gth,z)$ in the principal coordinates $\Gth$ and $z.$ Then, introducing the normal coordinate $t$,
we obtain the set of orthogonal curvilinear coordinates $(t,\Gth,z)$ on the patch given by
\[
\BR(t,\Gth,z)=\Br(z,\Gth)+t\Bn(z,\Gth),
\]
where $\Bn$ is the \emph{outward} unit normal. We choose a part $S_0\subset S$ of the patch that is given by
$S_0=\{(\Gth,z) \ : \ \Gth\in[0,\delta], \ z\in[0,\delta]\}$.
 Denote next
\[
A_{z}=\left|\frac{\partial \Br}{\partial z}\right|,\qquad A_{\Gth}=\left|\frac{\partial \Br}{\partial\Gth}\right|,
\]
the two nonzero components of the metric tensor of $S_0$ and the two principal curvatures by $\Gk_{z}$ and $\Gk_{\Gth}$.
The signs of $\Gk_{z}$ and $\Gk_{\Gth}$ are chosen such that $\kappa_{z}$ and $\kappa_{\Gth}$ are positive for a sphere.
Then the gradient of a vector field $\By=(y_t,y_\Gth,y_z)\in W^{1,p}(S_0^h,\mathbb R^3)$ has the form 
\begin{equation}
\label{3.14}
\nabla\By=
\begin{bmatrix}
  y_{t,t} & \dfrac{y_{t,\Gth}-A_{\Gth}\Gk_{\Gth}y_{\Gth}}{A_{\Gth}(1+t\Gk_{\Gth})} &
\dfrac{y_{t,z}-A_{z}\Gk_{z}y_{z}}{A_{z}(1+t\Gk_{z})}\\[3ex]
y_{\Gth,t}  &
\dfrac{A_{z}y_{\Gth,\Gth}+A_{z}A_{\Gth}\Gk_{\Gth}y_{t}+A_{\Gth,z}y_{z}}{A_{z}A_{\Gth}(1+t\Gk_{\Gth})} &
\dfrac{A_{\Gth}y_{\Gth,z}-A_{z,\Gth}y_{z}}{A_{z}A_{\Gth}(1+t\Gk_{z})}\\[3ex]
y_{ z,t}  & \dfrac{A_{z}y_{z,\Gth}-A_{\Gth,z}y_{\Gth}}{A_{z}A_{\Gth}(1+t\Gk_{\Gth})} &
\dfrac{A_{\Gth}y_{z,z}+A_{z}A_{\Gth}\Gk_{z}y_{t}+A_{z,\Gth}y_{\Gth}}{A_{z}A_{\Gth}(1+t\Gk_{z})}
\end{bmatrix}
\end{equation}
in the orthonormal basis $\Be_{t}$, $\Be_{\Gth}$, $\Be_{z},$ where $f,x=\partial_xf.$ We choose $\By=\BI+\epsilon\Bu,$ where 
\begin{equation}
\label{3.15}
\begin{cases}
u_t=W(\frac{\Gth}{\sqrt{h}},z)\\
u_\Gth=-\frac{t\cdot W_{,\Gth}\left(\frac{\Gth}{\sqrt h},z\right)}{A_\Gth{\sqrt h}}\\
u_z=-\frac{t\cdot W_{,z}\left(\frac{\Gth}{\sqrt h},z\right)}{A_z},
\end{cases}
\end{equation}
where $w$ is a smooth function compactly supported on the mid-surface $S_0.$ For the rotation $\BR=\BI$ and the vector field $\Bb=0,$ this choice will give equality in (\ref{2.3}) 
for every fixed $h>0$ by choosing $\epsilon>0$ small enough.

\end{proof}

\section*{Acknowledgements}
This material is based upon work supported by the National Science Foundation under Grants No. DMS-1814361.

\end{document}